\documentclass{amsart}
\usepackage{amsmath, amssymb,epic,graphicx,mathrsfs,enumerate}
\usepackage[all]{xy}

\usepackage{amsthm}
\usepackage{amssymb}
\usepackage{latexsym}
\usepackage{longtable}
\usepackage{epsfig}
\usepackage{amsmath}
\usepackage{hhline}
\usepackage{amscd}
\usepackage{newlfont}
 \DeclareMathOperator{\perm}{Sym}
 
\DeclareMathOperator{\aut}{Aut}

\DeclareMathOperator{\ran}{rank}
 \DeclareMathOperator{\frat}{Frat}
\DeclareMathOperator{\ssl}{SL} \DeclareMathOperator{\psl}{PSL}
\DeclareMathOperator{\asl}{ASL} 

\DeclareMathOperator{\sym}{Sym}

\DeclareMathOperator{\GL}{GL} 
\DeclareMathOperator{\ASL}{ASL}

\DeclareMathOperator{\alt}{Alt}
\DeclareMathOperator{\End}{End}

\DeclareMathOperator{\irb}{IrrB}
\DeclareMathOperator{\fit}{Fit}

\renewcommand{\emptyset}{\varnothing}

\newcommand{\gen}[1]{\ensuremath{\langle #1\rangle}}
\newtheorem{thm}{Theorem}[section]
\newtheorem{cor}[thm]{Corollary}
 \newtheorem{lemma}[thm]{Lemma}
\newtheorem{prop}[thm]{Proposition} 
 \newtheorem{defn}[thm]{Definition}
\newtheorem{question}[]{Question} 
\numberwithin{equation}{section}

\renewcommand{\footnote}{\endnote}
\newcommand{\ignore}[1]{}\makeglossary
\begin{document}
\bibliographystyle{amsplain}

	%	\subjclass{20P05}
	%	\keywords{groups generation; waiting time; Sylow subgroups; permutations groups}
	\title[Minimal invariable generating sets]{Minimal invariable generating sets}
	
	\author{Daniele Garzoni}
	\address{
		Daniele Garzoni\\ Universit\`a degli Studi di Padova\\  Dipartimento di Matematica \lq\lq Tullio Levi-Civita\rq\rq\\email: daniele.garzoni@phd.unipd.it}	

	\author{Andrea Lucchini}
	\address{
		Andrea Lucchini\\ Universit\`a degli Studi di Padova\\  Dipartimento di Matematica \lq\lq Tullio Levi-Civita\rq\rq \\email: lucchini@math.unipd.it}

	\begin{abstract}A subset $S$ of a group $G$ invariably generates $G$ if, when each element of $S$ is replaced by an arbitrary conjugate,
		the resulting set generates $G.$ An invariable generating set $X$ of $G$ is called minimal if no proper subset of $X$ invariably generates $G.$ We will address several questions related to the behaviour of minimal invariable generating sets of a finite group.
		\end{abstract}
	\maketitle
	
	\section{Introduction}

The research around invariable generation of groups followed initially two lines. The first was related to the probabilistic invariable generation of finite symmetric groups; the second was concerned with the study of fixed points in group actions. The first began with Dixon in \cite{dixon}, where the following definition was given. For a finite group $G$ and a subset $\{g_1, \ldots , g_d\}$ of $G$, we say that $\{g_1, \ldots , g_d\}$ invariably generates $G$ if
$\{g_1^{x_1}, \ldots , g_d^{x_d}\}$ generates $G$ for every choice of $x_i \in G$. The questions raised in \cite{dixon} received some attention: see \cite{lp}, \cite{PRR}, \cite{efg}, \cite{Br}.

The bridge between invariable generation and permutation groups follows easily from the definition. Indeed, saying that $X$  invariably generates $G$ amounts to saying that, for every transitive action of $G$ on a finite set with more than one element $\Omega$, there exists some element of $X$ acting without fixed points on $\Omega$. For results in this direction, see for instance \cite{W1}, \cite{fg1}, \cite{sha}.

While these two lines of research have studied invariable generation mainly for particular classes of groups, it was in \cite{ig} and \cite{ig-infinite} that a systematic approach to the subject was proposed. In \cite{ig} (and in \cite{gm} independently) it was shown, among other things, that finite simple groups are invariably generated by two elements. In \cite{ig-infinite} it was considered the problem of invariably generating infinite groups (with analogue definition). One interesting feature here is that there exist infinite groups that are not invariably generated by any of their subsets. These two papers inspired some research, such as \cite{dlis}, \cite{ACheboGen}, \cite{GG}, \cite{gem}.

In this paper we study various questions related to minimal invariable generating sets of finite groups. A (classical) generating set $X$ of $G$ is called minimal if no proper subset of $X$ generates $G$. We denote by $d(G)$ and $m(G)$, respectively, the smallest and the largest cardinality of a minimal generating set of $G.$ Analogously, an invariable generating set $X$ of $G$ is called minimal if no proper subset of $X$ invariably generates $G$. We use the notations $d_I(G)$ and $m_I(G)$ for the smallest and the largest cardinality of a minimal invariable generating set of $G$. We also write $\gen X_I = G$ to denote that $X$ invariably generates $G$.

It is interesting to compare, in various directions, generation and invariable generation. This is partly the purpose of the present paper. For this reason, before explaining in some detail the content of the paper we would like to spend some words about this comparison.

There is a well developed theory of generation of finite groups. On the one side of the story there are soluble groups. Gasch\"{u}tz \cite{Gas} gave a formula to compute the minimal number of generators of a finite soluble group in terms of certain \lq local\rq \ and \lq global\rq \ parameters associated to a chief series of the group. On the other side of the story there are nonabelian simple groups. It follows from the Classification of Finite Simple Groups that simple groups are generated by two elements. Various strengthenings of this statement have been studied over the last three decades. In a sense, it is possible to combine the two stories in order to obtain a theory of generation for all finite groups. This can be done using \lq crowns\rq  (see \cite{crowns}). The concept of crown was introduced by Gasch\"{u}tz in \cite{pref} for finite soluble groups. Later this notion has been generalised to all finite groups (see for example \cite{paz} and \cite{laf}).
 In his paper, Gasch\"{u}tz analyses the structure of the chief factors of a
soluble group $G$ as $G$-modules. Associated with an irreducible $G$-module $A,$ there exists
a section of the group, called the $A$-crown, which, viewed
as a $G$-module, is completely reducible and homogeneous with composition
length equal to the number of complemented chief factors $G$-isomorphic to $A$
in any chief series of $G.$

On the contrary, we believe that the theory of invariable generation of finite groups is still poor. There is some systematic way to work for soluble groups, again via the theory of crowns (see Section \ref{prelimi}), but outside the soluble world the available methods are unsatisfactory. The reason for this is that the concept of invariable generation is a subtle one.

For instance, an obvious feature of generation is the following: if $x$ and $y$ generate $G$, then $x$ and $xy$ also generate $G$, because we may obtain $y$ as $x^{-1}(xy)$. This apparently innocent property lurks behind more or less all results related to generation. Think for instance to the proof of the crucial result known as Gasch\"{u}tz Lemma \cite{Gas2}, or to the definition of the Product Replacement Algorithm \cite{pra}. The latter should remind us how this innocent property allows to create new generating tuples from old ones. The possibility to create new generating tuples is usually important in proofs involving counting arguments.

It is very easy to find examples, even in $\sym(3)$, showing that this innocent property fails for the invariable generation. This constitutes a serious obstacle for extending proofs from the classical to the invariable setting. Moreover, we are aware of no nontrivial way to produce new invariable generating tuples from old ones (except from conjugating the elements, but the tuples obtained in this way can hardly be considered as new ones).

%A slogan we may extract from the present paper is that in the soluble case one can sometimes remedy to the above problems by invoking the theory of crowns, which with some luck allows to reduce to questions of vector spaces and linear algebra --- the ideal %environment for generation.

All this said, and with this in mind, let us look closer at the content of the paper. Section \ref{prelimi} is devoted to introduce the background material, mainly about crowns in finite soluble groups.

Section \ref{sdue} deals with the problem of estimating $m_I(G)$. We will do this for finite soluble groups, showing that $m_I(G) = m(G)$. This inequality is no more true in the general case. We will prove that the difference $m(G) - m_I(G)$ can be arbitrarily large: this statement is somewhat opposite to the known fact (proved in \cite{ig} and \cite{dlis}) that $d_I(G) - d(G)$ can be arbitrarily large.
In particular, we will show that $m(G) - m_I(G)$ can be arbitrarily large for the family of symmetric groups, using however the Classification of Finite Simple Groups. On the other hand, in Section \ref{stre} we will construct an example for which the Classification Theorem is not needed, and where we can really bound the two numbers $m(G)$ and $m_I(G).$ 

No example is known of a finite group $G$ with $m_I(G)>m(G)$, so we leave open the question whether the inequality $m_I(G)\leq m(G)$ is true for every finite group.

It follows from a result in universal algebra, known as Tarski irredundant basis theorem, that for every positive integer $k$ with $d(G)\leq k\leq m(G)$, $G$ contains a minimal generating set of cardinality $k$. In Section \ref{squattro} we will try to apply Tarski's argument to the invariable generation. This however will produce only a weak result, and the natural generalisation of Tarski's theorem to the invariable setting remains an open problem. Nevertheless, we will prove that such generalisation holds in the soluble case.

A relevant role in the study of the generating properties of a finite group is played by the Frattini subgroup, since it coincides with the set of  \lq non-generating\rq \ elements. In Section \ref {scinque} we will introduce the analogue of the Frattini subgroup from the point of view of the invariable generation. This will allow us to properly state the results of Section \ref{ssette}.
\iffalse
In \cite{ak} the class of $\mathcal B$-groups was defined as the class of finite groups for which $d(G)=m(G)$. In Section \ref{ssette} we will define the $\mathcal B_I$-groups as the finite groups for which $d_I(G)=m_I(G)$. It is known by \cite{ak} that $\mathcal B$-groups are soluble. It turns out that they are also $\mathcal B_I$-groups. On the other hand, a $\mathcal B_I$-group needs be neither soluble nor a $\mathcal B$-group. We will classify the soluble $\mathcal B_I$-groups. Many $\mathcal B_I$-groups which are not $\mathcal B$-groups will be conne
\fi
\iffalse
In Sections \ref{ssette} and \ref{slast}, extending definitions from \cite{ak} we will study the $\mathcal B_I$-groups, i.e. the finite groups for which $d_I(G)=m_I(G)$, and the groups having the invariable basis property, i.e. the finite groups all of whose subgroups are $\mathcal B_I$-groups. We will discuss all the groups having the invariable basis property (in particular there are only four nonsoluble examples). On the other hand, we will only describe the structure of the soluble $\mathcal B_I$-groups; the unsoluble $\mathcal B_I$-groups remain largely unexplored. Soluble $\mathcal B_I$-groups will have connections with \lq secretive\rq \ $p$-groups, introduced in \cite{hid} with different purposes.
\fi

In Sections \ref{ssette} and \ref{slast}, extending definitions from \cite{ak} we will study the $\mathcal B_I$-groups, i.e. the finite groups for which $d_I(G)=m_I(G)$, and the groups having the invariable basis property, i.e. the finite groups all of whose subgroups are $\mathcal B_I$-groups. While we will discuss in great details all the groups having the invariable basis property (in particular there are only four nonsoluble examples), we will only describe the structure of the soluble $\mathcal B_I$-groups. These have connections with \lq secretive\rq \ $p$-groups, introduced in \cite{hid} with different purposes. On the other hand, the unsoluble $\mathcal B_I$-groups remain largely unexplored.

One last remark. The reader will have noted that most results of the present paper concern finite soluble groups. As we hinted above, our understanding of the invariable generation of this class of groups is better than in general. This depends on the fact that, with some luck, crowns allow to reduce to questions of vector spaces and linear algebra: the ideal environment for generation. However, the reader should be alerted that even for this class of groups the situation is not easy, and many apparently approachable questions still do not find an answer.

\section{Preliminaries}\label{prelimi}
 
We begin with an easy and well known lemma.

\begin{lemma}
\label{basico}
Let $G$ be a finite group, $X$ be a subset of $G$ and $N$ be an abelian normal subgroup of $G$. Let $\pi: G \rightarrow G/N$ denote the natural projection.
\begin{itemize}
\item[(i)] $X$ invariably generates $G$ if and only if $X \not\subseteq \cup_{g \in G} M^g$ for every maximal subgroup $M$ of $G$.
\item[(ii)] If $G$ is nilpotent, then $X$ invariably generates $G$ if and only it generates $G$.
\item[(iii)] If $\pi(X)$ invariably generates $G/N$, and $Y \subseteq N$ generates $N$ as a $G$-module, then $X \cup Y$ invariably generates $G$.
\end{itemize}
\end{lemma}

\begin{proof}
(i) follows immediately from the definition. (ii) follows from (i), since in a finite nilpotent group every maximal subgroup is normal. A proof of (iii) can be found for instance in {\cite[Lemma 2.10]{ig-infinite}}.
\end{proof}

In the rest of this section we shall review the notion and the properties of crowns. As we recalled in the introduction, this notion can be given for arbitrary finite groups. In this paper, however, we will use crowns only for soluble groups. 
For more details, see for instance \cite[Section 1.3]{classes}.
	
Let $G$ be a finite soluble group, and let $\mathcal V_G$ be a set
of representatives for the irreducible $G$-groups that are
$G$-isomorphic to a complemented chief factor of $G$. For $A \in
\mathcal V_G$ let $R_G(A)$ be the smallest normal subgroup contained
in $C_G(A)$ with the property that $C_G(A)/R_G(A)$ is
$G$-isomorphic to a direct product of copies of $A$ and it has a
complement in $G/R_G(A)$. The factor group $C_G(A)/R_G(A)$ is
called the $A$-crown of $G$, and it is the socle of $G/R_G(A)$. The positive integer
$\delta_G(A)$ defined by $C_G(A)/R_G(A)\cong_G A^{\delta_G(A)}$ is
called the $A$-rank of $G$ and it coincides with the number of
complemented factors in any chief series of $G$ that are
$G$-isomorphic to $A$. Moreover $C_G(A)/R_G(A)$ is complemented in $G/R_G(A)$, so that $G/R_G(A)\cong A^{\delta_G(A)}\rtimes H$ with $H\cong G/C_G(A).$

\begin{lemma}\label{corona}
	Let $G$ be a finite soluble group with trivial Frattini subgroup. There exist
	$A\in \mathcal V_G$ and a nontrivial normal subgroup $U$ of $G$ such that $C_G(A)=R_G(A) \times U.$ If $G$ is nonabelian then $A$ can be  chosen with the extra property of being a nontrivial $G$-module.
\end{lemma}

\begin{proof}
By {\cite[Lemma 1.3.6]{classes}}, there exists
	$A\in \mathcal V_G$ and a nontrivial normal subgroup $U$ of $G$ such that $C_G(A)=R_G(A)\times U.$ 
	Assume that $A$ is a trivial $G$-module. Then $G=C_G(A)=R_G(A) \times U$. Write $R_G(A) =H$, which is nontrivial if $G$ is nonabelian. In this case there exist a crown $C_H(B)/R_H(B)$ and a nontrivial normal subgroup $W$ of $H$ such that $C_H(B)=R_H(B)\times W.$ We have $C_G(B)=C_H(B)\times U$ and $R_G(B)=R_H(B)\times U$, so $C_G(B)=R_G(B)\times W.$ This means that we may consider $B$ in place of $A.$ It is possible that also $B$ is a trivial $G$-module. In that case $G=C_G(B)=R_G(B) \times W = R_H(B) \times U \times W$, and we can repeat the previous argument with $H$ replaced by $R_H(B)$. Continuing in this way, we obtain a nontrivial irreducible $G$-module satisfying our statement, except in the case when $G$ is abelian.	
\end{proof}

The following lemma will be applied several times. It says that essentially we need to care only of what happens modulo $U$ and modulo $R_G(A)$.

	\begin{lemma}{\cite[Lemma 4 and Lemma 12]{ACheboGen}}\label{ReU}
		\label{relativo}Assume that $G$ is a finite  group with trivial Frattini subgroup and let $C=C_G(A), R=R_G(A), U$ be as in the statement of Lemma \ref{corona}. If $K \leq G$ is such that $KU =
KR= G$, then $K = G$. In particular, for $g_1,\dots,g_t \in G$ if $\langle g_1U,\dots,g_tU\rangle_I=G/U$ and
		$\langle g_1R,\dots,g_tR\rangle_I=G/R,$ then
		$\langle g_1,\dots,g_t\rangle_I=G.$
	\end{lemma}

\iffalse

	\begin{prop}{\cite[Proposition 8]{igdl}}\label{matrici} Let $K$ be a finite soluble group and let $A$ be a faithful nontrivial irreducible $K$-module. We may consider $A$ as a vector space over the field $\End_K(A)$.
		Suppose that $\langle y_1,\dots,y_d\rangle_I=K.$
		Let $w_1,\dots,w_t\in A^\delta$ with
		$w_i=(w_{1,i},\dots,w_{\delta,i})$. Let $\pi_i:A^\delta \mapsto A$ be the canonical projection onto the $i$-th component:
		$\pi_i(v_1, \dots , v_u)=v_i.$
		 For $j\in \{1,\dots,\delta\}$, consider the vector
		$r_j=\big(\pi_j(w_1), \dots , \pi_j(w_t)\big)=(w_{j,1}, \dots, w_{j,t})\in A^{t}.$
		Then $y_1w_1, y_2w_2,\dots,y_t w_t$ invariably generate
		$A^\delta\rtimes K$ if and only if the vectors $r_1,\dots,r_\delta$ are linearly independent modulo
		$B=\{(u_1,\dots,u_t)\in A^{t}\mid  u_i \in [y_i,A], \ i=1, \dots , t\}.$
%		In particular, there exist some elements $w_1,\dots,w_d\in A^\delta$ such that $y_1w_1,\dots,y_tw_t$ invariably generate
		% $A^\delta\rtimes K$ if and only if
	%	$$u\leq nd-\dim W=\sum_i \dim_{\End_K(V)}C_V(y_i).$$
	\end{prop}
	
	\begin{cor}\label{corimp}
In the notation of the previous proposition,  for $1\leq i\leq t,$ let $A_i=[y_i,A]$ and  $B_i=A/A_i.$ We may consider $A, A_i, B_i$ as vector spaces over the field $F=\End_K(A).$ Let $b_{j,i}=\pi_j(w_i)+A_i\in B_i$ and, for $1\leq j\leq \delta,$ let $\rho_j=(b_{j,1}, \dots, b_{j,t})\in B_1\times \dots \times B_t.$ Then $y_1w_1,\dots,y_tw_t$ invariably generate
	$A^\delta\rtimes K$ if and only if the vectors $\rho_1,\dots,\rho_\delta$ are linearly independent.
\end{cor}

\fi

The following represents the main result for dealing with invariable generation of finite soluble groups. 

\begin{prop}{\cite[Proposition 8]{igdl}}\label{matrici} Let $K$ be a finite soluble group and let $A$ be a faithful nontrivial irreducible $K$-module. We may consider $A$ as a vector space over the field $F = \End_K(A)$.
		Suppose that $\langle y_1,\dots,y_t\rangle_I=K.$
		Let $\delta$ be a positive integer and let $w_1,\dots,w_t\in A^\delta$ with
		$w_i=(w_{1,i},\dots,w_{\delta,i})$. Consider the matrix $W$ whose $i$-th column is $w_i$: 
\[W=\begin{pmatrix} w_{1,1}&\cdots& w_{1,t}\\\vdots& &\vdots\\ w_{\delta,1}&\cdots& w_{\delta,t}
	  \end{pmatrix}.\]
                     Then $y_1w_1,\dots,y_t w_t$ invariably generate
		$A^\delta\rtimes K$ if and only if the rows of $W$ (seen as vectors of $A^t$) are linearly independent modulo
		$B=\{(u_1,\dots,u_t)\in A^{t}\mid  u_i \in [y_i,A], \ i=1, \dots , t\}.$
		
In particular, there exist elements $w_1,\dots,w_t\in A^\delta$ such that $y_1w_1,\dots,y_tw_t$ invariably generate
		$A^\delta\rtimes K$ if and only if
		$$\delta \leq nt-\dim B=\sum_{i=1}^t \dim_F C_A(y_i).$$
	\end{prop}

We restate this proposition in a slightly different form that will suit better our exposition.

\begin{cor}\label{corimp}
In the notation of the previous proposition,  for $1\leq i\leq t$ let $A_i=[y_i,A]$ and  $B_i=A/A_i.$ Again we consider $A, A_i, B_i$ as vector spaces over the field $F=\End_K(A).$ The entries of the $i$-th column of $W$ may be seen modulo $A_i$, that is, may be seen as elements of $B_i$. Let $Z$ denote this new matrix:
\[Z=\begin{pmatrix} w_{1,1} + A_1&\cdots& w_{1,t} + A_t\\\vdots& &\vdots\\ w_{\delta,1} + A_1&\cdots& w_{\delta,t} + A_t
	  \end{pmatrix} = : \begin{pmatrix} b_{1,1}&\cdots& b_{1,t}\\\vdots& &\vdots\\ b_{\delta,1}&\cdots& b_{\delta,t}
	  \end{pmatrix}.\]
Then $y_1w_1,\dots,y_tw_t$ invariably generate
	$A^\delta\rtimes K$ if and only if the rows of $Z$ (seen as vectors of $B_1 \times \cdots \times B_t$) are linearly independent.
\end{cor}

\section{Estimating $m_I(G)$}\label{sdue}
\begin{prop} \label{fac}
Let $G$ be a finite soluble group. There exists a minimal invariable generating set of cardinality $m=m(G).$
\end{prop}
\begin{proof}We argue by induction on $|G|.$ By \cite[Theorem 2]{min}, $m(G)$ coincides with the
	number of non-Frattini factors in a chief series of $G.$ Since $m(G)=m(G/\frat(G)),$ we may assume $\frat(G)=1.$ Let $N$ be a minimal normal subgroup of $G$ and let $H$ be a complement of $N$ in $G$. By induction there exist $m(G)-1$ elements $h_1,\dots,h_{m(G)-1}$ that form a minimal invariable generating set for $H$. If $n$ is a nontrivial element of $N,$ then by Lemma \ref{basico} $\{h_1,\dots,h_{m(G)-1},n\}$ is a minimal invariable generating set of $G.$
\end{proof}

This shows that, in the soluble case, $m(G) \leqslant m_I(G)$. The following proposition shows that the other inequality holds as well. Here we use all preliminaries on crowns introduced in Section \ref{prelimi}.

\begin{prop}\label{mmi}
Let $G$ be a finite soluble group and let $m=m(G).$ If $\{x_1,\dots,x_t\}$ is a minimal invariable generating set of $G,$ then $t\leq m.$
\end{prop}

\begin{proof}
	The statement is trivially true if $G$ is nilpotent since, as observed in Lemma \ref{basico}, in this case the notion of generation and invariable generation coincide. So we may assume that $G$ is soluble but not nilpotent. We prove our statement by induction on $|G|.$ We may assume  $\frat G=1.$ Choose a nontrivial $G$-module $A\in \mathcal V_G$ such that $R=R_G(A), U, C=C_G(A)$ satisfy
	the property described in Lemma \ref{corona}.

There exists a positive integer $\delta$ such that $U\cong_G A^\delta.$ 
	  By \cite[Theorem 2]{min}, $m=m(G)$ coincides with the
	number of non-Frattini factors in a chief series of $G$, hence $m=m(G/U)+\delta.$ Up to reordering the indices, we may assume that there exists $s\leq t$ such that $x_1,\dots,x_s$ is a minimal invariable generating set of $G$ modulo $U.$ By induction $s\leq m(G/U)=m-\delta.$

We work now in $\overline G=G/R$ and, for every $g\in G$, we set $\overline g=gR.$  We have $C/R=UR/R\cong U\cong A^\delta$
	and $G/R\cong C/R \rtimes H/R$
	where $K:=H/R$ acts in the same say on each of the $\delta$ factors of $C/R\cong A^\delta$ and this action
	is faithful and irreducible.  We may identify $\overline G$ with the semidirect product $A^\delta\rtimes K$ and we can write $\bar x_i=w_iy_i$ with $w_i\in U=A^\delta$ and $y_i\in K.$ Since $\langle x_1U,\dots,x_sU\rangle_I=G/U$ and $K\cong G/C$ is an epimorphic image of $G/U,$ we deduce that $\langle y_1,\dots,y_s\rangle_I=K.$

We want to apply Proposition \ref{matrici} and Corollary \ref{corimp}, and we employ the notations used there. Moreover, for $1\leq k\leq t$ we denote by $Z_{\text{rem}(k)}$ the matrix obtained by $Z$ removing the $k$-th column: 
\[Z_{\text{rem}(k)}=\begin{pmatrix} b_{1,1}&\cdots& b_{1,k-1}& b_{1,k+1}&\cdots& b_{1,t}\\\vdots & &\vdots & \vdots & & \vdots\\ b_{\delta,1}&\cdots& b_{\delta,k-1}& b_{\delta,k+1}&\cdots& b_{\delta,t}
	  \end{pmatrix} = : \begin{pmatrix} \rho_{1,k} \\ \vdots \\ \rho_{\delta,k} \end{pmatrix}\]
(here the $\rho_{i,k}$ are row vectors; same below with the $\sigma_{i,k}$), and with $Z_{\text{kee}(k)}$ the matrix obtained by $Z$ keeping only the first $k$ columns:
\[Z_{\text{kee}(k)}= \begin{pmatrix} b_{1,1}&\cdots& b_{1,k}\\\vdots& &\vdots\\ b_{\delta,1}&\cdots& b_{\delta,k}
	  \end{pmatrix} = : \begin{pmatrix} \sigma_{1,k} \\ \vdots \\ \sigma_{\delta,k} \end{pmatrix}.\]
Since $\langle w_1y_1,\dots,w_ty_t\rangle_I=\overline G\cong A^\delta \rtimes K,$ we have that the rows of $Z$ are linearly independent. On the other hand, since $\{x_1,\dots,x_t\}$ is a minimal invariable generating set of $G,$ if $s<k\leq t$ then
 $\langle  x_1,\dots, x_s, x_{s+1},\dots, x_{k-1}, x_{k+1},\dots, x_t\rangle_I\neq  G$. Therefore, by Lemma \ref{ReU} $\langle \bar x_1,\dots,\bar x_s,\bar x_{s+1},\dots,\bar x_{j-1},\bar x_{j+1},\dots,\bar x_t\rangle_I\neq \bar G$, and consequently
the rows of $Z_{\text{rem}(k)}$ are linearly dependent.

We claim that, for every $s \leq k < t,$ adding to $Z_{\text{kee}(k)}$ the $(k+1)$-th column increases dimension of the row space, that is,
\begin{equation}\label{cresce}\dim_F \langle \sigma_{1,k},\dots,\sigma_{\delta,k}\rangle < \dim_F \langle \sigma_{1,k+1},\dots,\sigma_{\delta,k+1}\rangle.
\end{equation} Indeed, if the dimension stays the same then the $(k+1)$-th column is useless, i.e., $\dim_F \langle \sigma_{1,t},\dots,\sigma_{\delta,t}\rangle = \dim_F \langle \rho_{1,k+1}\dots,\rho_{\delta,k+1}\rangle$. But the left-hand side is equal to $\delta$, while the right-hand side is strictly smaller than $\delta$: contradiction.
Hence the claim is proved.

Since $\dim_F \langle \sigma_{1,t},\dots,\sigma_{\delta,t}\rangle = \delta$ we deduce
from (\ref{cresce}) that $t-s\leq \delta,$ from which $t \leq s+ \delta \leq (m - \delta) + \delta = m$.
\end{proof}

Combining the previous two propositions we deduce the following result.

\begin{thm}
\label{cano}
If $G$ is a finite soluble group, then $m(G)=m_I(G).$
\end{thm}

%\begin{cor}Let $G$ be a finite soluble group and let $m=m(G).$ If $x_1,\dots,x_m$ is a minimal invariable generating set of $G,$ then there exists $g_1,\dots,g_m\in G$ such that $x_1^{g_1},\dots,x_m^{g_m}$ is a minimal generating set of $G.$
%\end{cor}

It is easy to find examples of (nonsoluble) groups for which Proposition \ref{fac} fails, namely, examples of groups $G$ for which $m_I(G) < m(G)$. For instance, $m(\alt(5))=3$ while $m_I(\alt(5))=2$ (this is because any invariable generating set of $\alt(5)$ must contain an element of order 5 and an element of order 3). On the other hand, we do not have examples of groups for which $m(G) < m_I(G)$, and we raise the following question:

\iffalse

\begin{question}
	Does there exist a finite unsoluble group with $m(G)= m_I(G)$?
\end{question}

\begin{question}
	Does there exist a finite group with $m(G) < m_I(G)$?
\end{question}

\fi

\begin{question} \label{qu1}
	For a finite group $G$, is it true that $m_I(G)\leqslant m(G)$?
\end{question}

%\begin{question}
%	For a finite unsoluble group $G$ is it true that $m_I(G) < m(G)$?
%\end{question}

%Notice that if the previous questions have positive answers then the equality $m(G) = m_I(G)$ characterises the finite soluble groups among all finite groups.

It seems that often, for a nonabelian finite simple group $G$, the strict inequality $m_I(G) < m(G)$ holds. Still, there exist infinitely many examples in which equality is attained. We postpone the proof of this fact in Section \ref{pis}, since in Section \ref{alto} we will introduce some terminology that will ease the exposition.

One could even ask whether the following strengthening of Question \ref{qu1} is true:
if $\{x_1,\dots,x_t\}$ is a minimal invariable generating set of $G,$ then there exist $g_1,\dots,g_m\in G$ such that $\{x_1^{g_1},\dots,x_m^{g_t}\}$ is a minimal generating set of $G.$ Although we are not able to exhibit a soluble counterexample, the following shows that the statement is not true in general.
\begin{lemma}Let $G=\alt(29)$ and consider the following three elements:
\begin{small}
$$
\begin{aligned}a=&(2,\!3,\!4)(5,\!6,\!7)(8,\!9,\!10,\!11,\!12,\!13,\!14,\!15,\!16,\!17,\!18)(19,\!20,\!21,\!22,\!23,\!24,\!25,\!26,\!27,\!28,\!29),\\
b=&(1,\!2)(3,\!4)(5,\!6,\!7,\!8,\!9,\!10,\!11,\!12,\!13,\!14,\!15,\!16,\!17,\!18,\!19,\!20,\!21,\!22,\!23,\!24,\!25,\!26,\!27,\!28,\!29),\\
c=&(1,\!2)(3,\!4,\!5,\!6,\!7,\!8)(9,\!10,\!11,\!12,\!13,\!14,\!15,\!16,\!17,\!18,\!19,\!20,\!21,\!22,\!23,\!24,\!25,\!26,\!27,\!28,\!29).
\end{aligned}
$$
\end{small}
The set $\{a,b,c\}$ is a minimal invariable generating set of $G$, but for every $x,y,z \in G$,
$\{a^x, b^y, c^z\}$ is not a minimal generating set.
\end{lemma}
\begin{proof}
	It can be easily seen that no proper subgroup of $\alt(29)$ contains conjugates of $a$, $b$ and $c,$ so $\langle a,b,c\rangle_I=\alt(29).$ On the other hand  $\langle a, b\rangle$ stabilises $\{1,2,3,4\},$ $\langle b, c\rangle$ stabilises $\{1,2\}$ and $\langle a^{(2,8)},c\rangle$ stabilises $\{3,4,5,6,7,8\}$ so $\{a,b,c\}$ is a minimal invariable generating set of $\alt(29).$ Now we want to show that, in any
	way we conjugate $a,b,c$, two elements are sufficient in order to generate $\alt(29).$
	Without loss of generality we may assume that one of this conjugates is $a.$ Let $x,y \in \alt(29).$ If $\langle a,b^x\rangle\neq \alt(29),$
	then $\langle a,b^x\rangle$ stabilises either
	$\{1,2,3,4\}$ (in which case $\{1,2,3,4\}$ is mapped into itself by $x$) or 	$\{1,5,6,7\}$ (in which case $\{1,2,3,4\}$ is mapped to 	$\{1,5,6,7\}$  by $x$).
	Without loss of generality we may assume that $a, b$ and $x$
	stabilise 	$\{1,2,3,4\}.$ If $	\langle a, c^y\rangle \neq \alt(29)$, then it stabilises $\{2,3,4,5,6,7\}$ (and the 6-cycle in the decomposition of $c^y$ permutes the elements of this subset). But then $\langle b^x, c^y\rangle=\alt(29)$.
Indeed two conjugates of $b$ and $c$ either generate $\alt(29)$ or stabilise the same subset of cardinality 2. But this second possibility does not occur for $b^x$ and $c^y,$ indeed the support of the 2-cycle in the decomposition of $b^x$ is contained in $\{1,2,3,4\}$ while the support of the 2-cycle in the decomposition of $c^y$ must be disjoint from $\{2,3,4,5,6,7\}.$
\end{proof}

If $G$ is  a finite group, then $d_I(G)\geq d(G)$ and the difference $d_I(G)-d(G)$ can be arbitrarily large. \cite[Proposition 2.5]{ig} states that, for every $r \geq 1,$ there is a finite group $G$ such that
$d(G) = 2$ but $d_I (G) \geq r.$ We do not know whether the (somewhat opposite) inequality $m(G)\geq m_I(G)$ is true, but in any case we may exhibit examples in which the difference
$m(G)-m_I(G)$ is arbitrarily large.

A first example is given by the symmetric group $\sym(n)$. It is immediate to check that $\{(1,2), (2,3), \dots, (n-1,n)\}$ is a minimal generating set for $\sym(n)$, from which $m(\sym(n)) \geq n-1$ (in fact, Whiston \cite{whi} showed that $m(\sym(n)) = n-1$). The set of transpositions above is far from being an invariable generating set, since all its elements are conjugate. It would be interesting to exhibit \lq elegant' and \lq large' minimal invariable generating sets for $\sym(n)$ (compare with \cite{Caca}, where it is shown that all minimal generating sets of maximal size for $\sym(n)$ are \lq elegant\rq).

In any case, an easy argument (see Section \ref{alto}) shows that, in every finite group $G$, $m_I(G)$ is at most the number of conjugacy classes of maximal subgroups of $G$. Using the Classification of Finite Simple Groups, Liebeck and Shalev \cite{lisha} showed that the number of conjugacy classes of maximal subgroups of $\sym(n)$ is of the form $(1/2+o(1))n$, from which we may deduce
\begin{thm}
$m(\sym(n)) - m_I(\sym(n)) \rightarrow \infty$ as $n \rightarrow \infty$.
\end{thm}

In the next section we will give a more elementary example. With this purpose, we recall that in \cite{min2} it is noticed that $m(A\times B)=m(A)+m(B)$ for every pair of finite groups $A$ and $B.$

\begin{question} Is it true that $m_I(A)+m_I(B)= m_I(A\times B)$ for every pair $(A,B)$ of finite groups?
\end{question}

It is easy to see that the inequality  $m_I(A)+m_I(B)\leq  m_I(A\times B)$ always holds. Indeed, if $\{a_1,\dots,a_r\}$ is a minimal invariable generating set of $A$ and $\{b_1,\dots,b_s\}$ is a minimal invariable generating set of $B,$ then $\{(a_1,1),\dots,(a_r,1),(1,b_1),\dots,(1,b_s)\}$ is a minimal invariable generating set of $A\times B.$ Regarding the equality, we are only able to prove a very partial result.

\begin{prop}Assume that $A$ and $B$ are finite groups without common composition factors. Then $m_I(A\times B)=m_I(A)+m_I(B).$
\end{prop}
\begin{proof}
	Assume that $g_1=(a_1,b_1),\dots,g_m=(a_m,b_m)$ is an invariable generating set of $G=A\times B.$ There exists $I\subseteq \{1,\dots,m\}$ such that $\{a_i \mid i \in I\}$ is a minimal invariable generating set for $A$ and $J\subseteq \{1,\dots,m\}$ such that $\{b_j \mid j \in J\}$ is a minimal invariable generating set for $B$. Then $\{(a_k,b_k)\mid k \in I\cup J\}$ is an invariable generating set for $A\times B.$ So $m_I(A\times B)\leq m_I(A)+m_I(B).$ 
\end{proof}

\section{An example: $m_I(\alt(5)^n)$}\label{stre}
\label{alto}	

Since $m(\alt(5))=3$ and $m(A\times B)=m(A)+m(B)$ for every pair of finite groups $A$ and $B$, we have $m(\alt(5)^n)=3n.$ We are going to show that $m(\alt(5)^n)-m_I(\alt(5)^n) \rightarrow \infty$ as $n \rightarrow \infty$. Indeed we shall prove:
\begin{prop}
	$m_I(\alt(5)^n)= n \cdot m_I(\alt(5)) = 2n$.
\end{prop}

Notice first that, by what we said in the previous section, $n \cdot m_I(\alt(5))\leq m_I(\alt(5)^n),$ so it suffices to show $m_I(\alt(5)^n)\leq 2n$. Let us introduce some considerations that we will employ to prove the previous proposition. A minimal invariable generating set of $G$ cannot contain two conjugate elements so it may be  identified with a subset of set $\mathcal C(G)$ of the conjugacy classes of $G$. For every maximal subgroup $M$ of $G$ denote by $M^*$ the subset of $\mathcal C(G)$ consisting of the conjugacy classes of $G$ with non-empty intersection with $M$. Let $C_1,\dots,C_t$ be a set of distinct conjugacy classes of $G$ and, for every $1\leq i\leq t,$ choose a representative $g_i\in C_i.$ We have that $\gen{g_1,\dots,g_t}_I=G$ if and only if $\{g_1,\dots,g_t\}\not\subseteq \cup_{g\in G}M^g$ (i.e.
$\{C_1,\dots,C_t\}\not\subseteq M^*$) for all maximal subgroups $M$ of $G$. Let $$\mathcal M(G)=\{M^*\mid M \text { a maximal subgroups of $G$}\}.$$ We say that a subset $\{X_1,\dots,X_t\}$ of $\mathcal M(G)$ is independent if, for every  $1\leq i\leq t,$ the intersection $\cap_{j\neq i}X_j$ is properly contained in $\cap_j X_j.$ We denote by $\iota(G)$ the largest cardinality of an independent subset of $\mathcal M(G).$
\begin{lemma}
\label{iota}
	$m_I(G)\leq \iota(G).$
\end{lemma}
\begin{proof}
Let $m=m_I(G)$ and let $\{x_1,\dots,x_m\}$ be a minimal invariable generating set of $G.$ For $1\leq i \leq m$ let $C_i$ be the conjugacy class of $G$ containing $x_i.$ For every $1\leq i \leq m,$ there exists a maximal subgroup $M_i$ of $G$ such that
$\{C_1,\dots,C_{i-1},C_{i+1},\dots,C_m\} \subseteq M_i^*$ but $C_i\notin M_i^*.$ It follows that $\{M_1^*,\dots,M_m^*\}$ is an independent subset of $\mathcal M(G),$ and therefore $m\leq \iota(G).$
\end{proof}

\begin{prop}
	$\iota(\alt(5)^n)\leq 2n.$
\end{prop}
\begin{proof}
	We have 5 conjugacy classes $C_1,C_2,C_3,C_4,C_5$ in $\alt(5)$ with representatives $1, (1,2)(3,4), (1,2,3), (1,2,3,4,5), (1,5,4,3,2).$ Notice that $C_1=C_1^{-1},$
	$C_2=C_2^{-1},$ $C_3=C_3^{-1},$ while $C_5=C_4^{-1}$
	and $C_4=C_5^{-1}.$
	Moreover a maximal subgroup of $\alt(5)$ is isomorphic to $\alt(4), \perm(3)$ or $D_{10}$ and $\mathcal M(\alt(5))$ contains only two elements: $Y_1=\{C_1,C_2,C_3\}$ and $Y_2=\{C_1,C_2,C_4,C_5\}.$
	Let $\Omega=\{C_1,C_2,C_3,C_4,C_5\}$, $\Omega^*=\{C_1,C_3,C_4,C_5\}$, $\Delta=\Omega^n$ and $\Delta^*=(\Omega^*)^n$. Notice that we are identifying
	the elements of $\Delta$	with the conjugacy classes of
$\alt(5)^n$
	.
	
	Let $G=\alt(5)^n.$ A maximal subgroup $M$ of $G$ can be of two different kinds:
	\begin{enumerate}
		\item there exist $1\leq i\leq n$ and a maximal subgroup $Y$ of $\alt(5)$ such that $(x_1,\dots,x_n)\in M$ if and only if $x_i\in Y$ (\slshape{product type}).
		\item there exist $1\leq i<j\leq n$ and $\phi\in \aut(\alt(5))$ such that $(x_1,\dots,x_n)\in M$ if and only if $x_j=x_i^\phi$ (\slshape{diagonal type}).
	\end{enumerate}
	As a consequence, the elements of $\mathcal M(G)$ are of the following kinds:	
	\begin{enumerate}
		\item $A_i=\{(\omega_1,\dots,\omega_n)\in \Delta \mid \omega_i\in Y_1\},$
		\item $B_i=\{(\omega_1,\dots,\omega_n)\in \Delta \mid \omega_i\in Y_2\},$
		\item $C_{i,j}=\{(\omega_1,\dots,\omega_n)\in \Delta \mid \omega_i=\omega_j\},$
		\item $D_{i,j}=\{(\omega_1,\dots,\omega_n)\in \Delta \mid  \omega_i=\omega_j^{-1}\}.$
	\end{enumerate}
	We now assume that $\{X_1,\dots,X_t\}$ is an independent subset of $\mathcal M(G)$ and we set
	$\Delta_i=X_1\cap \dots \cap X_i$, $\Delta_i^*=\Delta_i\cap 
	\Delta^*.$  Moreover let $\Lambda_i$ be the set of the $j\in \{1,\dots,n\}$ such that $\omega_j\notin \{C_4,C_5\}$ for every $(\omega_1,\dots,\omega_n)\in \Delta_i.$
	
	We may assume that there exist $a,b$ such that
	\begin{itemize}
		\item If $i\leq a$ then there exists $I_i=(r_i,s_i)$ such that
		$X_i\in \{C_{r_i,s_i},D_{r_i,s_i}\}.$
		\item If $a<i\leq a+b$ then $X_i=A_r$ for some $r$.
		\item If $a+b<i$ then $X_i=B_r$ for some $r$.
	\end{itemize}
%We can view to any element of $\Delta_i$ as an $n\times 5$ %matrix with entries in the set $\{C_1,C_2,C_3,C_4,C_5\}$.
For $i\leq a$, let $\rho_i$ be the smallest equivalence relation on $\{1,\dots,n\}$ containing all the pairs $(r_j,s_j)$ with $j\leq i.$ We may assume, up to reordering the indices, that there exists $a_1\leq a$ such that for every $2\leq i\leq a_1$ the relation $\rho_{i-1}$ is finer than $\rho_{i},$ while $\rho_i=\rho_{a_1}$ if $i>a_1.$
	We can describe how $\Delta_{a_1}$ looks like. Assume that $B_1,\dots,B_l$ are the equivalence classes of the relation $\rho_{a_1}.$ Then $\Delta_{a_1}$ 
	 is a product of $l$ \lq diagonal subsets\rq \, each of cardinality $5$: if $i_1,i_2\in B_j$ for some $1\leq j\leq l,$ then there exists $\epsilon_{i_1,i_2}=\pm 1$ such that $\omega_{i_2}=\omega_{i_1}^{\epsilon_{i_1,i_2}}$ for every $(\omega_1,\dots,\omega_n)\in \Delta_{a_1}.$
	 In particular, since $l\leq n-a_1$, we have
	 $$|\Delta_{a_1}|=5^l\leq 5^{n-a_1} \text { and } |\Delta^*_{a_1}|\leq {4^{n-a_1}}.$$ Now assume $a_1<i\leq a$. There exists an equivalence class $B_j$ of $\rho_{a_1}$ containing $r_i$ and $s_i$ and $\eta=\pm 1,$  that $\omega_{s_i}=\omega_{r_i}^\eta$ for every $(\omega_1,\dots,\omega_n)\in X_i.$ As we noticed above, there already exists  $\epsilon=\epsilon_{r_i,s_i}$ such that $\omega_{s_i}=\omega_{r_i}^{\epsilon}$ for every $(\omega_1,\dots,\omega_n)\in \Delta_{a_1}.$ We must have $\eta=-\epsilon$ (otherwise $\Delta_{a_1}\cap X_i=\Delta_{a_1}$), and consequently $\omega_{s_i}=\omega_{r_i}=\omega_{r_i}^{-1},$ (i.e. $\omega_{r_i} \notin \{C_4,C_5\}$) for every $(\omega_1,\dots,\omega_n)\in \Delta_i.$
	 In particular
	$$ |\Delta^*_i|\leq \frac {|\Delta_{i-1}^*|}{2}.$$
	Notice also that $|\Lambda_{a_1}|=0$ and $|\Lambda_a|\geq a_2,$ where we set $a_2=a-a_1.$
	
	Now assume $a<i \leq a+b$: again when we consider the intersection $\Delta_{i-1}\cap X_i$ we add the restriction that $\omega_i$ cannot belong to $\{C_4,C_5\}$, so $i\notin \Lambda_a$ (otherwise $\Delta_{a}\cap X_i=\Delta_{a_1}$) and $$ |\Delta^*_i|\leq \frac {|\Delta_{i-1}^*|}{2}.$$ Moreover $|\Lambda_{a+b}|\geq a_2+b.$
	
	Finally let $a+b<i.$ We may assume that there exists $c_1$ such that $X_i=B_r$ with $r \in \Lambda_{a+b}$ 	if and only if $i\leq a+b+c_1$. If $a+b<i\leq a+b+c_1,$ then
 $$|\Delta^*_i|\leq \frac {|\Delta_{i-1}^*|}{2}.$$ 
 %Finally for $a+b+c_1< i\leq a+b+c_1+c_2=t$ we have $$|\Delta^*_i|\leq \frac {3|\Delta_{i-1}^*|}{4}.$$ 
	We must have $$1\leq |\Delta^*_{a+b+c_1}|\leq  \frac{4^n }{4^{a_1}\cdot 2^{a_2+b+c_1}}$$ 
	and consequently $2a_1+a_2+b+c_1\leq 2n.$
Set $c_2=c-c_1$.  Notice that $a_2+b+c_2\leq n$ (since $c_2\leq |\{1,\dots,n\}\setminus \Lambda_{a+b}|\leq n-a_2-b$)	and $c_1+c_2\leq n$ (since there are at most $n$ maximal subgroups of kind $B_r$), hence $2c_2+a_2+b+c_1\leq 2n.$ But then $2t=(2a_1+a_2+b+c_1)+(2c_2+a_2+b+c_1)\leq 4n$, from which $t \leq 2n$.
\end{proof}

\section{$m_I(G)=m(G)$ with $G$ nonabelian simple}
\label{pis}
In this section we will exhibit infinitely many nonabelian finite simple groups $G$ for which $m_I(G)=m(G)$ holds.

\begin{prop} Assume $p$ is a prime such that the following conditions are both satisfied: $p \equiv 1$ mod $40$ and $p \equiv 2$ mod $3$. Then $m_I(\psl(2,p))=m(\psl(2,p))=3$.
\end{prop}
Notice that there exist infinitely many primes $p$ satisfying the conditions in the statement. Indeed, every prime $p \equiv 41$ mod $120$ satisfies them, and there exist infinitely many such primes by Dirichlet's theorem on arithmetic progressions. We remark that, with analogous proof, the statement holds also for $p \equiv -1$ mod $40$ and $p \equiv 1$ mod $3$.
We proceed with the proof of the proposition.

\begin{proof}
In \cite{SJ} it was shown that, for $p > 31$, $m(\psl(2,p))=3$, hence it remains to prove $m_I(\psl(2,p))=3$. Let $G_p=\psl(2,p)$. The subgroup structure of this group is well known, and we refer the reader to \cite[Chapter 3, Section 6]{Suzu} for detailed information. In particular, the condition $p \equiv 1$ mod $40$ implies that the isomorphism classes of maximal subgroups of $G_p$ are exactly the following: dihedral groups $D_{p-1}$ and $D_{p+1}$ of order $p-1$ and $p+1$, a Borel subgroup $B$ of order $p(p-1)/2$, $H=\alt(5)$ and $K=\sym(4)$.

Consider $X=\{x,y,z\}$ where $|x|=3$, $|y|=4$ and $|z|=5$. No proper subgroup of $G_p$ contains elements of order $3$, $4$ and $5$, hence $X$ is an invariable generating set (the conditions on $p$ imply that, while $B$ and $D_{p-1}$ contain elements of order $4$ and $5$, they do not contain elements of order $3$). Moreover, in $G_p$ every element of order coprime to $p$ can be conjugate inside a fixed dihedral group, hence whenever $|a|=|b| \geqslant 3$ and $\text{gcd}(p,|a|)=1$, we have $a^{G_p} \cap \gen b \neq \emptyset$.
Then, order considerations imply that any two elements of $X$ can be conjugate inside a suitable maximal subgroup of $G_p$. This shows that $m_I(G_p) \geqslant 3$.

For the other inequality, we employ the notation used in Section \ref{alto}. We will show $\iota(G_p) \leq 3$, so that $m_I(G_p) \leq 3$ by Lemma \ref{iota}. All subgroups isomorphic to $B$ are conjugate, and all involutions are conjugate, hence $\mathcal M(G_p)$ consists of $D_{p-1}^*, D_{p+1}^*, B^*, H^*, K^*$. We have that $B^* \cap D_{p+1}^*=D_{p-1}^* \cap D_{p+1}^*$ is the conjugacy class of involutions, which belongs to every member of the list. Moreover, $D_{p-1}^* \subseteq B^*$. This easily implies $\iota(G_p) \leq 3$.
\end{proof}

\section{The Tarski irredundant basis theorem}\label{squattro}

A nice result in universal algebra, due to Tarski and known with the name of  Tarski irredundant basis theorem (see \cite{Ta}, or \cite[Theorem 4.4]{bs}), implies that, for every positive integer $k$ with $d(G)\leq k\leq m(G)$, $G$ contains a minimal generating set of cardinality $k.$ A natural question is whether there exists a similar result for the invariable generation. Tarski's theorem relies on an elementary
but clever counting argument which is quite flexible and can be adapted to several different situations. However, as we shall see in this section, using this argument we are able to obtain only a weak and partial result. In order to see the problems in applying Tarski irredundant basis theorem to the invariable generation, we find it is interesting to sketch the proof of this partial result.

Tarski's theorem is based on the notion of closure operator (\cite[Definition 5.1]{bs}), which is a function $C$, from and to subsets of $G$, such that $X\subseteq C(X)$, $C(Y)\subseteq C(X)$ if $Y\subseteq X$, and $C(C(X)) = C(X)$. In case of generation, one defines $\mathcal C(X) = \gen X$. For the argument, it is important that $C(X)=G$ if and only if $X$ generates $G$ (this is obviously true in the case when we define $C(X)=\langle X\rangle).$ We should have this property also in the case of invariable generation. If $X = \{x_1, \ldots, x_t\}$, the first definition that comes to mind is then

\[C(X)=X \cup \left(\bigcap_{(g_1,\dots,g_t)\in G^t}\langle x_1^{g_1},\dots,x_t^{g_t}\rangle\right).\]

Artificially, we have imposed $X \subseteq C(X)$, and monotonicity is immediate. What is not immediate from the definition, but straightforward to check, is that $C$ is also idempotent. Moreover, it is not difficult to show that $C(X) = G$ if and only if $\gen X_I = G$. Therefore we have a closure operator, and we may be on the right track.

Now if we define, for $n,k \geq 1$,
\[C_n(X)=\bigcup_{Y\subseteq X, |Y|\leq n}C(Y),\quad  C_n^1(X)=C_n(X), \quad C_n^{k+1}(X)=C_n(C_n^k(X)),\]
following \cite{bs} we may call a finite group $G$ invariable $n$-ary if $C(X)=\bigcup_{i\in \mathbb N}C_n^i(X)$ for every subset $X$ of $G.$ Using this notion, it is possible to bound the \lq gap' that can occur between minimal invariable generating sets. More precisely, if we denote by $\irb_I(G)$ the set of the positive integers $n$ such that $G$ has a minimal invariable generating set of size $n$, we have the following
\begin{thm}
\label{nario}
	Let $G$ be an invariable $n$-ary finite group, with $n\geq 2.$
	If $i<j$ with $i,j \in \irb_I(G)$ such that $\{i+1,\dots,j-1\}\cap \irb_I(G)=\emptyset,$ then $j-i\leq n-1.$
\end{thm}
\begin{proof}
	Follows from the proof of \cite[Theorem 4.4]{bs}.
\end{proof}

\begin{cor}
	If $G$ is an invariable 2-ary finite group then, for every $d_I(G)\leq k \leq m_I(G),$ there exists a minimal invariable generating set of size $k.$
\end{cor}

Notice that a finite group $G$ is invariable 2-ary if the following holds: for every $X\subseteq G,$ if $C_2(X)=X$ then $C(X)=X.$

We see some problems in this approach. The first is that, although Theorem \ref{nario} does give a bound, we are not able to give a structural interpretation of the property of being invariable $n$-ary. Moreover, in case of nilpotent groups the closure operators defined for generation and for invariable generation need not coincide (remember that, instead, the notions of generation and invariable generation do coincide). Finally, the following result shows that Theorem \ref{nario} cannot give any absolute bound.

\begin{lemma} For every integer $n \geqslant 2$, there exists a finite group $G$ which is not invariable $n$-ary.
	\end{lemma}
\begin{proof} Let $n \geqslant 2$. Assume we prove that there exists a finite group $G$ with the following property: $d_I(G) \geqslant n+1$ and there exists $g \in G$ that does not lie in any proper normal subgroup of $G$. Then, if we set $X=G \setminus \{g\}$, we have that $\gen X_I=G$ (since $|G \setminus \cup_{g\in G}M^g| \geqslant |M|$ for every proper subgroup $M$ of $G$). Hence $C(X)=G \nsubseteq X$. On the other hand, for every $x_1, \ldots, x_n \in G$, $N:=\bigcap_{g_i \in G}\gen{x_1^{g_1}, \ldots, x_n^{g_n}}$ is a proper normal subgroup of $G$, hence $g \notin N$. This shows that $C_n(X) \subseteq X$, from which it follows that $G$ is not invariable $n$-ary.
	
We are left to exhibit a group with the property described above. For a supersoluble example, consider $G=P \rtimes Q$, where $P \cong C_p^n$ and $Q \cong C_q$ for primes $p$ and $q$, with $q$ dividing $p-1$, and $Q$ acts on each copy of $C_p$ as multiplication in the field $\mathbb{F}_p$. It can be easily seen
that $d_I(G)=n+1$. Moreover every proper normal subgroup of $G$ is contained in $P$, so we can take in the role of $g$ any element of $G\setminus P.$
	\end{proof}

Summarising, in order to apply Tarski irredundant basis theorem to the invariable generation we would need to define a closure operator $C$ on the set of subsets of $G$ with the following two properties:
\begin{enumerate}
	\item	$C(X)=G$ if and only if $\langle X \rangle_I=G,$
	\item $C(X)=\bigcup_{i\in \mathbb N}C_2^i(X)$ for every subset $X$ of $G.$
\end{enumerate}
We are not able to find a closure operation satisfying (1) different from the one introduced above. However this fails property (2). So the following question remains open.
\begin{question}
Let $G$ be a finite group. Is it true that for every $d_I(G)\leq k\leq m_I(G),$ there exists a minimal invariable generating set of $G$ of cardinality $k?$
\end{question}

We are able to prove that the answer is affirmative in the particular case of finite soluble groups.

\begin{thm}\label{tacchino}
	Let $G$ be a finite soluble group and let $m=m(G).$ If $\{x_1,\dots,x_t\}$ is a minimal invariable generating set of $G,$ with $t<m,$ then there exists a minimal invariable generating set of $G$ of cardinality $t+1.$
\end{thm}

\begin{proof}  The beginning of the proof is very similar to that of Proposition \ref{mmi}. As we did there, we prove our statement by induction on $|G|$. We may assume that $G$ is soluble but not nilpotent, the statement for nilpotent groups being easy to check (without applying Tarski's theorem). Again we may assume $\frat G=1,$ and we choose a nontrivial $G$-module $A\in \mathcal V_G$ such that $R=R_G(A), U, C=C_G(A)$ satisfy
	the property described in Lemma \ref{corona}. We let $\delta$ be such that $U\cong_G A^\delta$. By \cite[Theorem 2]{min}, $m=m(G/U)+\delta.$ Up to reordering the indices, we may choose $s\leq t$ such that $x_1,\dots,x_s$ is a minimal invariable generating set of $G$ modulo $U.$

In addition to what done in the proof of Proposition \ref{mmi}, we further choose, as we may, a complement $H$ of $U$ in $G$ with $R\leq H.$ For $1\leq i\leq t,$ we write $x_i=w_ih_i$ with $w_i\in U$ and $h_i\in H.$
	
	We work in $\overline G=G/R$ and, for every $g\in G$, we set $\overline g=gR.$  We may identify $\overline G$ with the semidirect product $A^\delta\rtimes K$ with $K= H/R.$ Since $\langle x_1U,\dots,x_sU\rangle_I=G/U$ and $K\cong G/C$ is an epimorphic image of $G/U,$ we deduce that $\langle \overline x_1,\dots,\overline x_s\rangle_I=K.$ 
	
	As in the proof of Proposition \ref{mmi}, we want to apply Proposition \ref{matrici} and Corollary \ref{corimp}, and we employ the notations used there. A technical, but important, step here is that we pick complements for $A_i$ in $A$. Namely, for $1\leq i \leq t$ we decompose $A$ as $A = A_i \oplus C_i$. Here the $C_i$ play the role of the $B_i$ in Proposition \ref{mmi} and Corollary \ref{corimp}, with the advantage (not apparent yet) that they are subspaces of $A$. Then, if we denote by $c_{j,i}$ the projection of $w_{j,i} \in A$ onto $C_i$, the matrix $Z$ of Corollary \ref{corimp} is replaced by
\[Z=\begin{pmatrix} c_{1,1}&\cdots& c_{1,t}\\\vdots& &\vdots\\ c_{\delta,1}&\cdots& c_{\delta,t}
	  \end{pmatrix}\]
Moreover, for $1\leq k\leq t$ the matrices $Z_{\text{rem}(k)}$ and $Z_{\text{kee}(k)}$ of Proposition \ref{mmi} become here
\[Z_{\text{rem}(k)}=\begin{pmatrix} c_{1,1}&\cdots& c_{1,k-1}& c_{1,k+1}&\cdots& c_{1,t}\\\vdots & &\vdots & \vdots & & \vdots\\ c_{\delta,1}&\cdots& c_{\delta,k-1}& c_{\delta,k+1}&\cdots& c_{\delta,t}
	  \end{pmatrix}\]

\[Z_{\text{kee}(k)}= \begin{pmatrix} c_{1,1}&\cdots& c_{1,k}\\\vdots& &\vdots\\ c_{\delta,1}&\cdots& c_{\delta,k}
	  \end{pmatrix}\]
As in Proposition \ref{mmi}, the rows of $Z$ (seen as vectors of $C_1 \times \cdots \times C_t \leq A^t$) are linearly independent, while, for every $s<k\leq t,$ the rows of $Z_{\text{rem}(k)}$ are linearly dependent.
	 
One observation. The definition of $C_i$, hence of the $c_{j,i}$, depend upon the element $h_i$. Then, once the $c_{j,i}$ have been defined, the $h_i$ have somewhat done their work (concerning generation modulo $R$), and we do not need to care about them anymore. Indeed, we only need to care of linear dependence of the rows of $Z$ inside $C_1 \times \cdots \times C_t$ or, equivalently, inside $A^t$. Then also the $C_i$ are not important anymore. This gives the possibility to suitably modify, to \lq clean' in some sense, the elements $x_i$ without affecting their property of invariable generation.

As a first example, if we denote by $\widetilde w_i \in A^{\delta} \cong U$ the $i$-th column of $Z$, we may replace $x_i$ with
             \[\widetilde x_i=\begin{cases}
	  \widetilde w_ih_i&\text{ if $i\leq s$,}\\
	  \widetilde w_i&\text{ otherwise.}
	  \end{cases}\]
It is easy to see that $\{\widetilde x_1,\dots \widetilde x_t\}$ is a minimal invariable generating set of $G.$ Indeed, by the choice of $s$ the set invariably generates modulo $U$ and it is not possible to remove one among the first $s$ elements. On the other hand, all the considerations regarding the invariable generation modulo $R$ are not affected, because they concern linear dependence of the rows of $Z$, and this matrix does not change in passing from $x_i$ to $\widetilde{x_i}$. This implies that $\{\widetilde x_1,\dots \widetilde x_t\}$ invariably generates $G$ minimally.

Now choose a subset $J=\{i_1,\dots,i_u\}$ of $\{1,\dots,s\}$
minimal with respect to the property that the $\delta$ vectors
	  \[(c_{j,i_1},\dots,c_{j,i_u},c_{j,s+1},\dots,c_{j,t})\] for $1 \leq j \leq \delta$ are linearly independent. The arguments applied in the previous paragraph imply that we still obtain a minimal invariable generating set if we replace $\widetilde x_j=\widetilde w_jh_j$ with $h_j$ for every $j\in \{1,\dots, s\}\setminus J.$ So from now on we will assume $\widetilde w_j=0$ for every  $j\in \{1,\dots, s\}\setminus J.$

  If $J\neq \emptyset,$ then we obtain a minimal invariable generating set of size $t+1$ by replacing $\widetilde  x_{i_1}=\widetilde w_{i_1}h_{i_1}$ with the two elements 
	  $\widetilde w_{i_1}$ and $h_{i_1}.$ 
	  
	  So we may assume $J=\emptyset$, from which it follows that the first $s$ columns of $Z$ are zero. For convenience, we may remove from the matrix $Z$ such columns. The rank of the matrix clearly does not change. We call the matrix obtained in this way again $Z$.
\[Z=\begin{pmatrix}c_{1,s+1}&\cdots&c_{1,t}\\\vdots& &\vdots\\ c_{\delta,s+1}&\cdots& c_{\delta,t}
	  \end{pmatrix}\]
For $s < k \leq t$, we remove the first $s$ columns in  $Z_{\text{rem}(k)}$ and $Z_{\text{kee}(k)}$. Again, the rows of $Z$ are linearly independent (i.e. $\ran Z = \delta$), while for $s < k \leq t$ the rows of $Z_{\text{rem}(k)}$ are linearly dependent (i.e. $\ran Z_{\text{rem}(k)} < \delta$).

Now the same argument as in the end of Proposition \ref{mmi} shows that for $s \leq k < t$, $\ran Z_{\text{kee}(k)}<\ran Z_{\text{kee}(k+1)}$. Let 
\[n_1=\ran Z_{\text{kee}(s+1)},\ n_2=\ran Z_{\text{kee}(s+2)}-\ran Z_{\text{kee}(s+1)},\dots, \]
\[
n_{t-s}=\ran Z_{\text{kee}(t)}- \ran Z_{\text{kee}(t-1)}.\]

Notice that $n_1+\cdots+n_{t-s}=\delta$, and $1 \leq n_i\leq n$ for every $i$. Let now $F=\End_H(A)$ and $n=\dim_F A.$ Fixing a basis for $A$ as an $F$-vector space, we may identify each element of $A$ as a vector of $F^n$. Denote by $e_i$ the vector of $F^n$ all of whose entries are 0, expect the $i$-th which is $1$, and consider the block
	   matrix
	   \[Y=\begin{pmatrix}
	   e_1&0&\cdots&0\\
	   \vdots&\vdots& &\vdots\\
	     e_{n_1}&0&\cdots&0\\
	    0&e_1&\cdots&0\\
	    \vdots&\vdots& &\vdots\\
	    0&e_{n_2}&\cdots&0\\
	      \vdots&\vdots& &\vdots\\
	      0&0&\cdots&e_1\\
	       \vdots&\vdots& &\vdots\\
	        0&0&\cdots&e_{t-s}
	   \end{pmatrix}
	   \]
Using the definition of the $e_i$, it is easy to check that we still obtain a minimal invariable generating set if we replace $Z$ with $Y$. More precisely, if we consider the $i$-th column of $Y$ as an element $\widetilde y_{s+i}$ of $U \cong A^{\delta}$, we get that  $\{\widetilde x_1,\dots,\widetilde x_{s},\widetilde y_{s+1} ,\dots,\widetilde y_t\}$ is a minimal invariable generating set of $G.$

Assume first that there exists $i\in \{1,\dots,s-t\}$ with
$n_i> 1.$ Then $\tilde y_{s+i} \in A^{\delta}$ has at least two nonzero entries, and it suffices to split $\tilde y_{s+i}$ in two vectors: if we define
$$\begin{aligned}
\widetilde z_1&=(0,\dots,0,e_1,\dots,e_{n_i-1},0,0,\dots 0),\\
\widetilde z_2&=(0,\dots,0,\ 0\ ,\dots,\ 0\ ,e_{n_i},0,\dots 0),\\
\end{aligned}$$ then the set $\{\widetilde x_1,\dots,\widetilde x_{s},\widetilde y_{s+1} ,\dots,\widetilde y_t\} \cup \{\widetilde z_1,\widetilde z_2\} \setminus \{\widetilde y_{s+i}\}$ is a minimal invariable generating set of size $t+1.$

Assume finally $n_i=1$ for every $i\in \{1,\dots,t-s\}$. In this case $t-s=\delta$. Since $t < m=m(H)+\delta,$ we get $s<m(H).$
Then, by induction, there exists a minimal invariable generating set $\{\widetilde k_1,\dots,\widetilde k_{s+1}\}$ of $H$ of cardinality $s+1$. It follows that $\{\widetilde k_1,\dots,\widetilde k_{s+1},\widetilde x_{s+1},\dots,\widetilde x_t\}$ is a minimal invariable generating set of $G$ of cardinality $t+1.$
	   \end{proof}

\section{The invariable Frattini}\label{scinque}
The Frattini subgroup $\frat G$ of a finite group $G$ is defined as the intersection of all maximal subgroups of $G$. An important feature of this subgroup is that it coincides with the elements of $G$ that are useless in generating $G$. More precisely, $\frat G$ coincides with the set of elements of $G$ that can be dropped from every generating set of $G$ (without compromising generation). This feature implies that the generation properties of $G$ are essentially the same as those of $G / \frat G$. Therefore, if we are interested in generation we can factor out $\frat G$ with no harm. This considerably simplifies the situation, since the structure of Frattini-free groups is much more transparent than that of general groups (at least for soluble groups: think of how many times we applied Lemma \ref{corona} and Lemma \ref{relativo}).

Here we shall define the analogue of the Frattini subgroup from the point of view of the invariable generation. This will allow us to properly state the results of Section \ref{bigr}. For every  subgroup $M$ of $G$, set $\widetilde M=\bigcup_{g\in G}M^g$. Consider the set $\Sigma=\Sigma(G)$ of all maximal members of the set
of all $\widetilde H,$ where $H$ varies among the proper subgroups of $G.$ Set $\frat_I(G)=\bigcap_{\widetilde M\in \Sigma}\widetilde M.$
\begin{lemma}\label{polpo}
	$\frat_I(G)$ coincides with the set of elements of G that can be dropped
	from any invariable generating set.
\end{lemma}
\begin{proof} Assume $x \in \frat_I(G)$ and assume $x \cup X$ invariably generates $G$ for
	some set $X.$ If $X$ does not invariably generate $G$ then $X \subseteq \widetilde M$ for some $\widetilde M \in \Sigma,$ hence $x \cup X \subseteq \widetilde M,$ against the assumption of invariable generation.
	Conversely, assume $x\notin \frat_I(G)$: choose $\widetilde M$ such that $x\notin \widetilde M.$ Then, by
	the maximality of $\widetilde M$ it follows that $\gen{x \cup \widetilde M}_I = G,$ and clearly $x$ cannot be
	omitted from this invariable generating set.
\end{proof}

By the previous lemma, $\frat_I(G)$ plays, for the invariable generation, the same role played by the Frattini subgroup for the usual generation. Unfortunately $\frat_IG$ need not be a subgroup. For instance, if $G=\alt(5)$ then $\frat_I(G)$ is the set of all involutions of $G$ -- hence it generates $G$.

Notice that if $\widetilde K\in \Sigma (G)$, then  $K$ is a maximal subgroup of $G$ and clearly if $M$ is a maximal subgroup of $G,$ then there exists a maximal subgroup $K$ of $G$ such that $\widetilde K\in \Sigma(G)$ and
$M\subseteq \widetilde M\subseteq \widetilde K,$ hence, by definition, $\frat G \subseteq \frat_I G$. This, if we want, is the reason why we can factor out $\frat G$ also in the invariable setting.

Notice that $\frat_I G$ is defined in a strange manner. Indeed, we do not intersect the $\widetilde M$'s for $M$ running among all maximal subgroups of $G$; we take instead only the maximal sets among the $\widetilde M$'s. This is important for the proof of Lemma \ref{polpo}. However, we do not know whether this is really necessary, and we propose the following

\begin{question}
\label{faina}
For a finite group $G$, does $\frat_I G$ coincide with the intersection of all $\widetilde M$, where $M$ runs among all maximal subgroups of $G$?
\end{question}

What we do know is that the two concepts are different a priori, meaning that there may exist maximal subgroups $M_1$ and $M_2$ such that $\widetilde{M_1}$ is properly contained in 
$\widetilde{M_2}.$ For example in $G=\alt(6)$ one can consider $M_1\cong \perm(4)$ and $M_2\cong 3^2:4 \cong (\sym(3) \wr \sym(2)) \cap \alt(6)$. Then $\widetilde M_2$ is the set of the elements of $G$ of order different from 5, while $\widetilde M_1$ does not contain elements of order 5 and moreover contains only one of the two conjugacy classes of elements of order $3$. Hence $\widetilde{M_1} \subsetneq \widetilde{M_2}$. Nevertheless, once again this phenomenon cannot occur in the soluble world.
\begin{lemma}
	Assume that $G$ is a finite soluble group and let $M_1, M_2$ be two maximal subgroups of $G$. If $\widetilde  M_1\subseteq \widetilde M_2,$ then $\widetilde  M_1 = \widetilde M_2.$ 
\end{lemma}

\begin{proof}
We prove the statement by induction on the order of $G.$  We may assume  $\frat G=1$. Choose a nontrivial $G$-module $A\in \mathcal V_G$ such that $R=R_G(A), U, C=C_G(A)$ satisfy
the property described in Lemma \ref{corona}. We further choose a complement $H$ of $U$ in $G$ with $R\leq H.$
We denote by $\mathcal M_1$ the set of the maximal subgroups of $G$ containing $U$ and by $\mathcal M_2$ the set of the maximal subgroups of $G$ supplementing $U.$ If $M\in \mathcal M_2$ then, by Lemma \ref{relativo}, $R\subseteq M$ and $M = WH^u$ with $W$ a
maximal $H$-submodule of  $U$ and $u\in U.$ Assume now $\widetilde  M_1\subseteq \widetilde M_2.$  We consider the different cases:
\begin{enumerate}
	\item $M_1, M_2 \in \mathcal M_1$. In this case $\widetilde {M_1/U}\subseteq \widetilde{M_2/U},$ so by induction $\widetilde {M_1/U}=\widetilde{M_2/U},$ and consequently $\widetilde {M_1} = \widetilde{M_2}.$
		\item $M_1, M_2 \in \mathcal M_2$.  We have $M_1=W_1H^{u_1}$ and $M_2=W_2H^{u_2}.$ If $W_1\neq W_2,$ then
		$\gen{M_1^{g_1},M_2^{g_2}}=G$ for every $g_1, g_2 \in G$, hence we cannot have neither the inclusion $\widetilde  M_1\subseteq \widetilde M_2$ nor the inclusion $\widetilde  M_2\subseteq \widetilde M_1.$ If $W_1=W_2$ then $M_1$ and $M_2$ are conjugates and $\widetilde M_1=\widetilde M_2.$
	\item $M_1\in \mathcal M_1$ and $M_2\in \mathcal M_2.$
	In this case $\gen{M_1^{g_1},M_2^{g_2}}=G$ for every $g_1, g_2 \in G$ and, as above, we cannot have neither $\widetilde  M_1\subseteq \widetilde M_2$ nor $\widetilde  M_2\subseteq \widetilde M_1.$ \qedhere
\end{enumerate}
\end{proof}

In particular, it follows from the previous lemma that Question \ref{faina} has an affirmative answer in case of finite soluble groups. 

We make another little regression before going on with the next, more substantial, section. It is well known that if a prime $p$ divides the order of a finite group $G$, then it divides also the order of $G/\frat(G)$. In particular, $G \setminus \frat(G)$ contains elements whose order is divisible by $p$. The analogue statement for invariable generation is false in general. For instance, if $G=\text{Alt}(5)$ then $G \setminus \frat_I(G)$ does not contain elements whose order is divisible by $2$.

Notice that in the case of classical generation we can say a little more, namely, we can say that $G \setminus \frat(G)$ contains elements of $p$-power order. This follows from the fact that it is always possible to lift an element without affecting the set of prime divisors of its order. For soluble groups, the corresponding \lq invariable' statement is true as well, although for the proof we invoke Hall's theorems.

\begin{lemma}
	\label{pote}
	Let $G$ be a finite soluble group. If a prime $p$ divides $|G|$, then the set $G \setminus \frat_I(G)$ contains elements of $p$-power order.
\end{lemma}
\begin{proof} 
	Consider a chief series of $G$, choose a nontrivial element from every complemented chief factor, and lift it to an element of $G$ of prime power order. It is easy to check that these elements together form an invariable generating set; we may therefore extract a minimal invariable generating set $X$. If $X$ did not contain any element of $p$-power order, then Hall's theorems  would imply $X \subseteq \widetilde K$, where $K$ is a Hall $p$-complement, contradicting the fact that $\gen{X}_I=G$.
\end{proof}

We apply this to prove a lemma that we will need in the following section. Unless otherwise stated, here and in the following sections modules are written multiplicatively, so that $1$ denotes the identity element.

%\begin{lemma}{\cite[Corollary 1.73, p. 21]{K}}
	%\label{Khu}
	%Let $V$ be an abelian $p$-group. If $\phi$ is an automorphism of $V$ of $p$-power order, then $C_V(\phi) \neq 1$.
%\end{lemma}

\begin{lemma}\label{nove} Let $H$ be a finite soluble group, and let $V$ be an $H$-module of finite $p$-power order. If $C_V(h)=1$ for every $h \in H \setminus \frat_I(H)$, then $p$ does not divide $|H|$.
\end{lemma}
\begin{proof}
	Assume by contradiction that $p$ divides $|H|$. Then, by Lemma \ref{pote} there exists $h \in H \setminus \frat_I(H)$ of $p$-power order. Now we may construct $G=V \rtimes \gen h$. This is a finite $p$-group, hence $V \cap Z(G) \neq 1$, from which $C_V(h) \neq 1$, contradicting the hypothesis.
\end{proof}

\section{$\mathcal B_I$-groups}\label{ssette}
\label{bigr}

A finite group $G$ is called a $\mathcal B$-group if $d(G)=m(G).$ The letter $\mathcal B$ refers to the word \lq basis\rq, \ since the property $d(G) = m(G)$ is a fundamental one for finite dimensional vector spaces. A classification of the Frattini-free $\mathcal B$-groups is given in \cite[Theorem 1.4]{ak}: $G$ is a Frattini-free $\mathcal B$-group if and only if one of the following holds:
\begin{enumerate}
	\item $G$ is an elementary abelian $p$-group for some
	prime $p;$ 
	\item $P\rtimes Q,$
	%$P\rtimes Q,$ 
	where $P$ is an elementary
	abelian $p$-group and $Q$ is a nontrivial cyclic $q$-group, for distinct primes $p\neq q,$ such that $Q$ acts faithfully on $P$ and the $Q$-module $P$ is a direct sum of $m(G)-1$ isomorphic copies of one simple module.
\end{enumerate}
We may give a similar definition for the invariable generation: a finite group $G$ is called a $\mathcal B_I$-group if $d_I(G)=m_I(G).$ It turns out that $\mathcal B$-groups are $\mathcal B_I$-groups (we include this statement in Proposition \ref{strutturabi} below).  Indeed, $\mathcal B$-groups are soluble. Moreover, $m(G)=d(G)\leq d_I(G)\leq m_I(G)=m(G)$, where the last equality follows from Theorem \ref{cano} (one can also check directly that the groups in (1) and (2) are $\mathcal B_I$-groups).

The converse implication is false. For example $d_I(\alt(5))=m_I(\alt(5))=2,$
so $\alt(5)$ is a $\mathcal B_I$-group but not a  $\mathcal B$-group. Another example is the following. Since $\alt(5)\cong \ssl(2,4)$ we may consider $G=\asl(2,4)\cong V\rtimes \alt(5)$, where $V$ is a 2-dimensional vector space over the field $\mathbb F_4$ with four elements. The elements of order 3 and 5 in $\alt(5)$ act fixed-point-freely on $V$, so if $g\in G$ either $|g|$ divides 4 or $g$ is conjugate to an element of order 3 or 5 in $\alt(5).$
If $X$ is an invariable generating set of $G$, then $X$ contains necessarily an element of order 3, an element of order 5 and a 2-element with a nontrivial power in $V$; but three elements of this kind invariably generate $G$, so $d_I(G)=m_I(G)=3.$

In this section we want to study the structure of soluble $\mathcal B_I$-groups. First notice that there exist soluble $\mathcal B_I$-groups that are not $\mathcal B$-groups. Indeed the quaternion group $Q_8$ is isomorphic to an irreducible subgroup of $\GL(2,3)$ and we may consider $G=V\rtimes Q_8$ where $V$ is a 2-generated vector space over the field $\mathbb F_3.$ The action of $Q_8$ on $V$ is fixed-point-free, which implies that no element of $G$ has order 6, so an invariable generating set of $G$ must contain two elements of order 4 and one element of order 3, and consequently $d_I(G)=3=m_I(G).$

It turns out that the soluble $\mathcal B_I$-groups which are not $\mathcal B$-groups are, in a sense, generalisations of the above example.

\begin{lemma}\label{prelim}
	Assume that $H$ is a finite soluble group and that $N$ is a faithful irreducible $H$-module. Then $G=N\rtimes H$ is a $\mathcal B_I$-group if and only if the following conditions hold:
	\begin{enumerate}
		\item $H$ is a $\mathcal B_I$-group.
		\item  $C_N(h)=1$ for every $h\in H\setminus \frat_I(H).$
	\end{enumerate}
\end{lemma}
\begin{proof} Notice that if $\langle h_1,\dots,h_t\rangle_I=H$ and $1\neq n\in N$, then by Lemma \ref{basico} $\langle h_1,\dots,h_t,n\rangle_I=G.$ Moreover, by Proposition \ref{matrici}, there exist $n_1,\dots,n_t \in N$ such that $\langle h_1n_1,\dots,h_tn_t\rangle_I=G$ if and only if $C_N(h_i)\neq 1$ for some $1\leq i\leq t.$ So $G$ is a $\mathcal B_I$-group if and only if all the minimal invariable generating sets of $H$ have the same cardinality (i.e. $H$ is a $\mathcal B_I$-group) and whenever an element $h$ of $H$ appears in some minimal invariable generating set of $H$ (i.e. whenever $h\not\in \frat_I(G)),$ then $C_N(h)=1.$
\end{proof}

%Accidentally, in the following we will essentially reprove {\cite[Theorem 1.4]{ak}}.

\begin{prop}\label{strutturabi}
	Let $G$ be  a finite soluble group. Then $G$ is a $\mathcal B_I$-group if and only if one of the following occurs:
	\begin{enumerate}
		\item $G$ is a $\mathcal B$-group;
		\item $G/\frat(G)\cong N\rtimes H$ where $H$ is a $\mathcal B_I$-group, $N$ is a faithful irreducible $H$-module and $C_N(h)=1$ for every $h\in H\setminus \frat_I(H).$ In particular, by Lemma \ref{nove}, $H$ and $N$ have coprime orders.
	\end{enumerate}
\end{prop}

\begin{proof}We may assume $\frat(G)=1.$
	Let $m=m_I(G)=m(G)$ and $F=\fit G.$ We have that $F$ has a complement $H$ in $G$ and that $F=N_1\times \dots \times N_t$ where $N_i$ is an irreducible $H$-module. First we claim that $N_i\cong_G N_j$ for every $i\neq j.$ Indeed assume for example $N_1\not\cong_GN_2.$ Choose $1\neq x_1\in N_1$ and $1\neq x_2\in N_2$ and let $x=x_1x_2.$ Take a set $\{y_1,\dots,y_{m-2}\}$ of invariable generators of $G$ modulo $N_1N_2$ and consider $X=\{y_1.\dots,y_{m-2},x\}.$ Assume that there exist $g_1,\dots,g_{m-2},g \in G$ such that $Y:=\langle y_1^{g_1},\dots,y_{m-2}^{g_{m-2}},x^g\rangle\neq G.$ It follows that $Y$ is a common complement of $N_1$ and $N_2$, but this implies $N_1\cong_G N_2$, a contradiction. So $\langle X \rangle_I=G$ and $d_I(G)\leq m-1<m=m_I(G),$ against the assumption that $G$ is a $\mathcal B_I$-group. So our claim has been proved and we may assume $F\cong_G N^t$ for a suitable irreducible $G$-module $N.$
	Let $K=\End_H N$ and $n=\dim_{K}N.$ Recall that $F=C_G(F)$ and $G/F$ is isomorphic to a subgroup of $\GL(q,n)$, being $q=|K|.$ There are three cases:
	\begin{enumerate}
		\item[(a)] $N$ is central. In this case $G=F$ is  an elementary abelian $p$-group.
		\item[(b)] $N$ is non central and $n=1$: in this case $G/F$ is cyclic: but then $m_I(G/F)=d_I(G/F)=1$, hence $G/F$ is a $q$-group for some prime $q$ not dividing $|N|.$ We conclude that $G=F\rtimes Q,$
		%$P\rtimesQ,$ 
		where $F$ is an elementary
		abelian $p$-group and $Q$ is a nontrivial cyclic $q$-group, for distinct primes $p\neq q,$ such that $Q$ acts faithfully on $F$ and the $Q$-module $F$ is a direct sum of $m(G)-1$ isomorphic copies of one simple module.
		\item[(c)] $N$ is non central and $n\neq 1.$ We claim that this implies $t=1.$ Indeed $G=N^t\rtimes H$ satisfies the hypothesis of Proposition \ref{matrici}. Suppose $t\neq 1$, let $\{y_1,\dots,y_{m-t}\}$ be an invariable generating set of $H$ and take $y_{m-t+1}=\dots=y_{m-1}=1.$ Since 
		\[\quad\quad\quad
\sum_{1\leq i\leq m-1} \dim_K C_N(y_i)\geq\! \!\!\sum_{m-t<i\leq m-1}\!\!\! \dim_K C_N(y_i)=n(t-1)\geq 2(t-1)\geq t,
\]
there exist $w_1,\dots,w_{m-1}\in N^t$ such that $G=\langle y_1w_1,\dots,y_{m-1}w_{m-1}\rangle_I,$ but then $d_I(G)\leq m-1 <m_I(G),$ a contradiction.
	\end{enumerate}
	In cases (a) and (b) $G$ is a $\mathcal B$-group, and it was already observed that $\mathcal B$-groups are $\mathcal B_I$-groups. In case (c) we may apply Lemma \ref{prelim} to conclude that $G$ is a $\mathcal B_I$-group if and only if the condition in (2) is satisfied. 
\end{proof}
To construct  $\mathcal B_I$-groups that are not $\mathcal B$-groups, we have to look for non-cyclic
$\mathcal B$-groups $H$ admitting a faithful irreducible $H$-module $N$ with the property that $C_N(h)=1$ for every $h\in H\setminus \frat_I(H)$, and construct then $G=N\rtimes H.$

For example, the dicyclic group $H$ of order 12 is a $\mathcal B_I$-group and has an irreducible and fixed-point-free action on the 2-dimensional vector space $V$ over the field with 13 elements,
so $N\rtimes H$ is a $\mathcal B_I$-group of order $12\cdot 13^2$ which is not a $\mathcal B$-group.

We can also take $H$ to be a non-cyclic $p$-group. In this case, however, the only possibility to have an irreducible fixed-point-free action is when $p=2$ and $H$ is a generalised quaternion group \cite[10.5.5]{rob}. If we want examples with $p$ odd, we need 
finite $p$-groups $H$ with an irreducible action on a module $N$ which is not fixed-point-free, but such that $C_N(y)=1$ for every $y\notin \frat H = \frat_I H$.

Interestingly, the $p$-groups with this property have been studied with different purposes. They have been called \lq secretive\rq \ in \cite{hid}. Wall \cite{wa} proved that for each prime $p$ and integer $d \geq 2$ there exists a finite secretive $p$-group $P$ with $d(P)=d$. Therefore we have several examples of soluble $\mathcal B_I$-groups which are not $\mathcal B$-groups.

Outside the soluble case we know almost nothing. The problem of investigating the finite unsoluble $\mathcal B_I$-groups is entirely open.

\section{Invariable basis property}\label{slast}

A group $G$ has the basis property if and only if $d(H)=m(H)$ for every $H\leq G.$ The groups with this property are classified in \cite[Corollary A.1]{ak}. In a similar way we can say that $G$ has the invariable basis property if $d_I(H)=m_I(H)$ for every $H\leq G.$ If $G$ has the invariable basis property, then every cyclic subgroup of $G$ has prime-power order. The groups all of whose elements have prime-power order are called CP-groups. They are studied in \cite{hh}.
\begin{lemma}\label{somma}
	Let $G$ be a finite group and let $N$ be a soluble normal subgroup of $G$. Denote by $t$ the number of non-Frattini factors lying below $N$ in a chief series passing through $N.$ If $g_1N,\dots,g_dN$ is a minimal invariable generating set of $G/N$, then $G$ admits a minimal invariable generating set of cardinality $d+t.$ 
\end{lemma}
\begin{proof}
	The proof is by induction on $t$, so it suffices to prove this statement in the particular case when $N$ is a non-Frattini minimal normal subgroup of $G.$ In this case there exists a complement $H$ of $N$ in $G$. For every $i,$ we can write $g_i=h_in_i$ with $h_i\in H$ and $n_i\in N.$ For $1\neq n\in N,$ by Lemma \ref{basico} $\{h_1,\dots,h_d,n\}$ is a minimal invariable generating set of $G.$
\end{proof}

\begin{lemma}\label{ere}
	Let $G$ be a finite group and let $N$ be a soluble normal subgroup of $G$. If $G$ has the invariable basis property, then $G/N$ also has the invariable basis property.
\end{lemma}
\begin{proof}
	Follows immediately from Lemma \ref{somma}.
\end{proof}

\begin{prop}\label{partial}
	Suppose that $G$ is a finite soluble group with $\frat(G)=1.$  Then $G$ satisfies the invariable basis property if and only if one of the following occurs.
	\begin{enumerate}
		\item $G$ is an elementary abelian $p$-group.
		\item $G=P\rtimes Q,$
		%$P\rtimes Q,$ 
		where $P$ is an elementary
		abelian $p$-group and $Q$ is a nontrivial cyclic $q$-group, for distinct primes $p\neq q,$ such that $Q$ acts faithfully on $P$ and the $Q$-module $P$ is a direct sum of isomorphic copies of one simple module.
		\item $G=N\rtimes H,$ where $H$ is a generalised quaternion group,  the action of $H$ on $N$ is irreducible and $|N|=p^2$ where $p$ is a prime with $p\equiv 3\mod 4.$ In this case $H$ coincides with the Sylow 2-subgroup of $\ssl(2,p).$
	\end{enumerate}
\end{prop}
\begin{proof}
	$G$ is in particular a $\mathcal B_I$-group, so it satisfies one of the two possibilities described in Proposition \ref{strutturabi}. If $G$ is a $\mathcal B$-group, then $G$ satisfies (1) or (2). Otherwise $G=N\rtimes H$ where $N$ is $H$-irreducible, $\dim_{\End_H(N)}N \neq 1$, $C_N(h)=1$ for every $h\in H\setminus \frat_I(H)$ and $(|H|,|N|)=1.$ Moreover $G$ is a soluble CP-group so, by 
	\cite[Theorem 1]{hig}, $G$ has order divisible by at most 2 primes. Since $(|H|,|N|)=1,$ we conclude that  $H$ has prime power order. Since every element of $G$ has prime power order, we also deduce that $H$ acts fixed-point freely on $N.$ By \cite[10.5.5]{rob}, $H$ is cyclic or generalised quaternion. However we may exclude the first case, since it implies  $\dim_{\End_H(N)}N = 1.$
	
	Let us first consider the case $H=Q_8$, the quaternion group. Assume $|N|$ is a power of $p,$ being $p$ an odd prime.  Let $F_p$ be the field with $p$ elements.
	We have, up to equivalence, a unique faithful irreducible $F_pQ_8$-representation, say $\phi_p,$ and this representation has degree 2. Indeed choose $a, b$ in $F_p$ such that $a^2+b^2=-1.$ Then 
	$\phi_p: Q_8\to \GL(2,F_p)$ is defined by setting
	$$\phi_p(i)=\begin{pmatrix}a&b\\b&-a\end{pmatrix},\quad
	\phi_p(j)=\begin{pmatrix}0&-1\\1&0\end{pmatrix},\quad \phi_p(k)=\begin{pmatrix}b&-a\\-a&-b\end{pmatrix}.$$
	Since $$\phi_p(-1)=\begin{pmatrix}-1&0\\0&-1\end{pmatrix},$$
	$Q_8$ acts fixed-point-freely on $N=F_p^2$ and $G=N\rtimes Q_8$ is a $\mathcal B_I$-group. Notice that $\phi_p(i), \phi_p(j)$ and $\phi_p(k)$ have minimal polynomial $x^2+1.$

If $p\equiv 1 \mod 4$, then there is $c\in F_p$ such that $c^2=-1$, hence we may choose $(a,b)=(c,0)$ and $\phi_p(i)$ has eigenvalues $c$ and $-c.$  In this case consider $X=N \rtimes \langle i \rangle$, where $N=\langle w_1, w_2\rangle$ with $w_1^i=cw_1$ and $w_2^i=-cw_2$. We have $X=\langle w_1+w_2,i \rangle_I$, so $2=d_I(X)<m_I(X)=3$ and $G$ does not satisfy the invariable basis property. It follows that $p\equiv 3 \mod 4$.
	
	Consider now the general case
	$G=V\rtimes Q_{2^n},$ where $V$ is an elementary abelian $p$-group and $Q_{2^n}$ is the generalised quaternion group
	$$Q_{2^n}=\langle x,y \mid x^{2^{n-1}}=1, y^2= x^{2^{n-2}},
	y^{-1}xy=x^{-1}\rangle.$$ Suppose that this  group has the invariable basis property. In particular $K=\langle x^{2^{n-3}},y \rangle \cong Q_8$ is a subgroup of $Q_{2^n}$ and 
	$V\rtimes K$ is a subgroup of $G$. Since $G$ has the invariable basis property, $V\rtimes K$ is a $\mathcal B_I$-group, hence by Proposition \ref{strutturabi} $V$ is a faithful irreducible $K$-module. It follows that $|V|\leq p^2$ and $Q_{2^n}$ can be identified with a subgroup of $\GL(2,p).$ Since $V\rtimes K$ has the invariable basis property, we conclude again $p\equiv 3 \! \mod 4.$ Moreover, $y$ is an element of order 4 of $\GL(2,p)$, hence its characteristic polynomial is $t^2+1$ and consequently $\det y=1.$ Let $\alpha$ and $\beta$ be the eigenvalues of $x$ (in the algebraic closure of $F_p$). Since $x$ and $x^{-1}$ are similar matrices we have $\{\alpha, \beta\}=\{\alpha^{-1}, \beta^{-1}\}$, from which $\beta=\alpha^{-1}$ and $\det x=1.$ Hence $Q_{2^n}\leq \ssl(2,p).$ We deduce from \cite[Chap. 2, Theorem 8.3 (ii)]{gor} that $Q_{2^n}$ coincides with a Sylow $p$-subgroup of $\ssl(2,p).$

	Conversely, it is not difficult to see that if $G$ satisfies (1), (2) or (3), then $G$ has the invariable basis property.
	\end{proof}

\begin{cor}\label{invbasol}Let $G$ be a finite soluble group with the invariable basis property. Then $G=P\rtimes Q,$ where $P$ and $Q$ are Sylow subgroups of $G$ and the action of $Q$ on $P$ is fixed-point-free. In particular $Q$ is cyclic or generalised quaternion.
\end{cor}
\begin{proof}Let $F=\frat(G).$ By Proposition \ref{partial},
	$G/F=X\rtimes Y$ where $X$ is a $p$-group, $Y$ is a $q$-group and $p$ and $q$ are distinct primes. Since $F$ is nilpotent and contains no element of order $p\cdot q,$ we deduce that $F$ is either a $p$-group or a $q$-group. Assume by contradiction that $F$ is a nontrivial $q$-group and let $P$ be a Sylow $p$-subgroup of $G$. Clearly $FP$ is a normal subgroup of $G$, so by the Frattini argument, $G=FPN_G(P)=FN_G(P)=N_G(P).$ But then both $P$ and $F$ are normal in $G$, so $[P,F]=1$ and $G$ contains an element of order $p\cdot q,$ a contradiction. Therefore $F$ is a $p$-group and the statement follows.
\end{proof}

While we did not study unsoluble $\mathcal B_I$-groups, the invariable basis property is restrictive enough to allow, with the help of the results in \cite{hh}, a characterisation of all groups having this property. In particular, there are only four unsoluble groups sharing it.

\begin{lemma}
\label{bricco}
	Let $G$ be a nonabelian finite simple group. Then $G$ has the invariable basis property if and only if it is isomorphic to one of the following:
	\begin{enumerate}
		\item $L_2(5),$ $L_2(8).$
		\item $Sz(8),$ $Sz(32).$
	\end{enumerate}
\end{lemma} 

\begin{proof}
	$G$ must be a CP-groups, so by \cite[Proposition 3]{hh}
	$G$ is isomorphic to one of the following:
	\begin{enumerate}
		\item $L_2(q)$ for $q=5,7,8,9,17.$
		\item $L_3(4).$
		\item $Sz(8), Sz(32)$
	\end{enumerate}
	However, $L_2(7), L_2(9), L_2(17)$ and $L_3(4)$ have a subgroup isomorphic to $\perm(4)$: since $2=d_I(\perm(4))<m_I(\perm(4))=3,$ these groups do not have the invariable basis property. We analyse the remaining cases:
	\begin{itemize}
		\item $G=L_2(5) \cong \alt(5)$. We have already noticed that $d_I(G)=m_I(G)=2$. It can be easily seen that if $H$ is a proper subgroup of $G$ then either $H$ is a $p$-group or $H$ is non cyclic with $m_I(H)=2$. Hence $G$ has the invariable basis property.
		\item $G=L_2(8).$ An element of $G$ can have order 1, 2, 3, 7, 9 and there are three conjugacy classes of maximal subgroups: $F_{56},$ $D_{18},$ $D_{14}.$ The minimal invariable generating sets of $G$ are precisely the sets consisting of two elements, one of order 7, the other of order 3 or 9, so $d_I(G)=m_I(G)=2.$ It can be easily seen that if $H$ is a proper subgroup of $G$ then either $H$ is a $p$-group or $H$ is non cyclic with $m_I(G)=2.$
		\item $G=Sz(8).$ An element of $G$ can have order 1, 2, 4, 5, 7, 13 and there are four conjugacy classes of maximal subgroups: $2^{3+3}\!:\!7$ (the Frattini subgroup has order 8, and the factor group over the Frattini subgroup has a unique minimal normal subgroup, of order 8), $13\!:\!4,$ $5\!:\!4,$ $D_{14}.$ The minimal invariable generating sets of $G$ are precisely the sets consisting of two elements $x, y$ such that $\{|x|,|y|\}=
		\{4,7\}, \{5,7\}, \{5,13\}$ or $\{7,13\}$. Again it can be easily seen that if $H$ is a proper subgroup of $G$ then either $H$ is a $p$-group or $H$ is non cyclic with $m_I(G)=2.$
		\item $G=Sz(32).$ An element of $G$ can have order 1, 2, 4, 5, 25, 31, 41 and there are four conjugacy classes of maximal subgroups: $2^{5+5}\!:\!31$ (the Frattini subgroup has order 32, and the factor group over the Frattini subgroup has a unique minimal normal subgroup, of order 32), $41\!:\!4,$ $25\!:\!4$ (the Frattini subgroup has order 5), $D_{62}.$ The minimal invariable generating sets of $G$ are precisely the sets consisting of two elements $x, y$ such that $\{|x|,|y|\}=
		\{5,31\}, \{25,31\}, \{25,41\}$ or $\{31,45\}$. Again it can be easily seen that if $H$ is a proper subgroup of $G$ then either $H$ is a $p$-group or $H$ is non cyclic with $m_I(G)=2.$\qedhere
	\end{itemize}
\end{proof}

	\begin{cor}
		Let $G$ be a finite nonsoluble group. Then $G$ has the invariable basis property if and only if $G \in \{L_2(5), L_2(8), Sz(8), Sz(32)\}.$
	\end{cor}
	\begin{proof}
		We have to prove only the direct implication. $G$ is a CP-group so by \cite[Proposition 2]{hh}, there are normal subgroups $1\leq N\leq M\leq G$ of $G$ such that  $G/M$ is soluble, $M/N = S$ is a finite nonabelian simple group and $N$ is a 2-group. By Lemmas \ref{ere} and \ref{bricco}, $M/N \in \{L_2(5), L_2(8), Sz(8), Sz(32)\}$; we want to show $M=G$ and $N=1$. 

It follows from Propositions 4 and 5 in \cite{hh} that $M=G.$
		%Assume, by contradiction, $N\neq 1.$ Let $L$ be a normal subgroup of $G$ such that $N/L$ is a chief factor of $G.$ Let $S=G/N$ and $V=N/L$ (so in particular) $V$ is an irreducible 2-module.
		Notice that $S$ contains a subgroup isomorphic to the dihedral group of order $2\cdot p$, with $p=5$ if $S=L_2(5),$ $p=7$ if $S\in \{L_2(8), Sz(8)\},$ $p=31$ if $S=Sz(32).$ So there exists a subgroup $H$ of $G$ containing $N$ and with the property that $H/N\cong D_{2p}.$ Since $H$ satisfies the invariable basis property, we deduce from Corollary \ref{invbasol} that $H$ has a normal Sylow $p$-subgroup, say $P$, and consequently $N\leq C_G(P)$. Since $G$ cannot contain elements of order $2\cdot p$, we conclude $N=1.$
	\end{proof}

\end{document}